\documentclass{daj}

\usepackage{amsmath}
\usepackage{amsfonts}
\usepackage{stmaryrd}
\usepackage{amssymb}
\usepackage{amsthm}
\usepackage{csquotes}

\usepackage{mathrsfs}
\usepackage{dsfont}
\usepackage{bm}
\usepackage{cite}
\usepackage{soul}

\numberwithin{equation}{section}
\renewcommand\vec{\bm}
\newcommand{\n}[1]{\|{#1}\|}

\newtheorem{theorem}{Theorem}[section]

\newtheorem{lemma}[theorem]{Lemma}
\newtheorem{Proposition}[theorem]{Proposition}
\newtheorem{Conjecture}[theorem]{Conjecture}
\newtheorem{Corollary}[theorem]{Corollary}

\usepackage{mathtools}

\dajAUTHORdetails{%
  title = {New Lower Bounds for Cardinalities of Higher Dimensional Difference Sets and Sumsets}, 
  author = {Akshat Mudgal},
  plaintextauthor = {Akshat Mudgal},
  runningtitle = {On Higher Dimensional Difference Sets and Sumsets}, 
  keywords = {Additive combinatorics, Asymmetric sumsets, Difference sets },
}   

\dajEDITORdetails{%
   year={2022},
   number={15},
   received={29 October 2021},   
   published={1 December 2022},  
   doi={10.19086/da.55552},       
}   

\begin{document}

\begin{frontmatter}[classification=text]


\author[am]{Akshat Mudgal\thanks{Supported by Ben Green's Simons Investigator Grant, ID 376201}}

\begin{abstract}
Let $d \geq 4$ be a natural number and let $A$ be a finite, non-empty subset of $\mathbb{R}^d$ such that $A$ is not contained in a translate of a hyperplane. In this setting, we show that
\[ |A-A| \geq \bigg(2d - 2 + \frac{1}{d-1} \bigg) |A| - O_{d}(|A|^{1- \delta}), \]
for some absolute constant $\delta>0$ that only depends on $d$. This provides a sharp main term, consequently answering questions of Ruzsa and Stanchescu up to an $O_{d}(|A|^{1- \delta})$ error term. We also prove new lower bounds for restricted type difference sets and asymmetric sumsets in $\mathbb{R}^d$.
\end{abstract}
\end{frontmatter}

\section{Introduction}

In this paper, we study a classical problem in additive combinatorics concerning lower bounds for cardinalities of difference sets in higher dimensions. Given a natural number $d$ and finite, non-empty sets $A$ and $B$ of $\mathbb{R}^d$, we define the sumset $A+B$ and the difference set $A-B$ as
\[ A+ B = \{ a + b \ | \ a \in A, \ b \in B \}\]
and
\[A- B = \{ a- b \ | \ a \in A, \ b \in B\}. \]
We also define the \emph{dimension} of a non-empty set $A \subseteq \mathbb{R}^d$, denoted by $\dim(A)$, to be the dimension of the affine subspace spanned by $A$. If $\dim(A) = k$ for some $k \in \mathbb{N}$, we say that $A$ is a $k$-\emph{dimensional set}. 
\par

The study of higher dimensional sumsets has been a major topic of research in additive combinatorics. For instance, we note a classical result of Freiman, popularly known as Freiman's lemma (see \cite[Lemma $5.13$]{TV2006}), that states that whenever $A$ is a finite, non-empty subset of $\mathbb{R}^d$ satisfying $\dim(A) =d$, we have
\begin{equation} \label{i1frl}
 |A+A| \geq (d+1)|A| - d(d+1)/2. 
 \end{equation}
Furthermore, this can be seen to be sharp by noting that the set
\begin{equation} \label{trin1}
 A_N = \{0, e_1, \dots, e_{d-1}\} + \{n \cdot e_d \ | \ 1  \leq n \leq N\}
 \end{equation}
satisfies 
\[ |A_N + A_N| \leq (d+1)|A_N| - O_{d}(1), \]
where we use $\{e_1, \dots, e_d\}$ to denote the canonical basis of $\mathbb{R}^d$ and we write $\lambda \cdot \vec{v}= (\lambda v_1, \dots, \lambda v_d)$ for any $\lambda \in \mathbb{R}$ and $\vec{v} = (v_1, \dots, v_d) \in \mathbb{R}^d$. 
\par

While sharp estimates for cardinalities of sumsets of the form $A+A$ have therefore been well known, there have been no known corresponding results for $A-A$ that have been sharp in dimensions $d \geq 4$. Questions concerning such lower bounds have been asked by various authors, including Ruzsa \cite{Ru1994} and Stanchescu \cite{St2001}. Moreover, estimates of this form have been applied to improve results in geometry of numbers, including the classical theorem of Minkowski-Blichfeld (see \cite{Ak2021} for references). Till recently, the only known lower bound for $|A-A|$, when $A$ is chosen to be a $d$-dimensional set for some arbitrary $d \in \mathbb{N}$, was given by work of Freiman--Heppes--Uhrin \cite{FHU1989}, who showed that
\begin{equation} \label{inm}
  |A-A| \geq (d+1)|A| - d(d+1)/2. 
  \end{equation}
While this can be seen as an analogue of $\eqref{i1frl}$ for difference sets and it is, in fact, optimal when $d \in \{1,2\}$, this was not known to be sharp for $d \geq 3$, since even the set $A_N$, that optimised $\eqref{i1frl}$, satisfies $|A_N - A_N| = (2d - 2 + 2/d)|A_N| - O_{d}(1)$ for every $d \in \mathbb{N}$.
\par

The only other regime where a sharp lower bound is known to hold for this problem is the case when $d=3$, where Stanchescu \cite{St1998} proved that whenever $\dim(A) = 3$, we have $|A-A| \geq (4 + 1/2)|A| - 9$, which answered a question of Ruzsa \cite{Ru1994}. Furthermore, this can be seen to optimal from later work of Stanchescu \cite{St2001}, where for each $d \geq 3$ it was shown that there exist $d$-dimensional sets $B_i \subseteq \mathbb{R}^d$, with $|B_i| \to \infty$ as $i \to \infty$, such that
\begin{equation} \label{con78}
 |B_i - B_i| \leq (2d - 2 + 1/(d-1))|B_i| - (2d^2 - 4d + 3). 
 \end{equation}
Stanchescu \cite{St2001} further conjectured that this was the extremal example for every $d \geq 4$.

\begin{Conjecture} \label{con45}
Let $A$ be a finite, non-empty subset of $\mathbb{R}^d$ such that $\dim(A) = d$. Then
\[ |A-A| \geq (2d - 2 + 1/(d-1))|A| - (2d^2 - 4d + 3). \]
\end{Conjecture}

The best known lower bound for this problem when $d \geq 4$ was recently upgraded by one of our previous results \cite{Ak2021}, where we showed that there exists a constant $\delta >0$, depending on $d$ only, such that
\[ |A-A| \geq (2d-2)|A| - O_{d}(|A|^{1-\delta}), \]
for every $d$-dimensional subset $A$ of $\mathbb{R}^d$. This improved upon the aforementioned result of Freiman--Heppes--Uhrin \cite{FHU1989} and delivered an estimate for $|A-A|/|A|$ that was within $1/(d-1)$ of the main term given by Conjecture $\ref{con45}$.
\par

In this paper, we record a further improvement in this direction, obtaining the bound of Conjecture $\ref{con45}$ up to an $O_{d}(|A|^{1- \delta})$ error term.

\begin{theorem} \label{diff}
Let $d \geq 4$ and let $A$ be a finite, non-empty subset of $\mathbb{R}^d$ such that $\dim(A) = d$. Then we have
\[ |A-A| \geq \bigg(2d - 2 + \frac{1}{d-1}\bigg) |A| - O_d(|A|^{1- \delta}), \]
for some absolute constant $\delta >0$ that only depends on $d$.
\end{theorem}

When $d \geq 4$, this improves upon our result in \cite{Ak2021} and answers the questions of Ruzsa \cite{Ru1994} and Stanchescu \cite{St2001} concerning cardinalities of difference sets, up to an $O_d(|A|^{1 - \delta})$ error term. Moreover, these lower bounds are sharp up to the $O_d(|A|^{1 - \delta})$ error term, as Stanchescu's example $\eqref{con78}$ shows. 

Our proof of Theorem $\ref{diff}$ entails proving various other results concerning cardinalities of difference sets, the foremost of these being an estimate about lower bounds for restricted type difference sets. We present this below.

%

\begin{theorem} \label{diffln}
Let $d \geq 1$ and let $A$ be a finite, non-empty subset of $\mathbb{R}^d$ such that $\dim(A) = d$. Moreover, suppose that $A$ is supported on $r$ translates of the line $L_{\vec{v}} = \{ \lambda \cdot \vec{v} \ | \ \lambda \in \mathbb{R}\}$ for some $\vec{v} \in \mathbb{R}^d \setminus \{0\}$. Then we have
\[ |(A-A)\setminus L_{\vec{v}}| \geq (2d - 2) |A| - 2d^2 r. \]
\end{theorem}

Theorem $\ref{diffln}$ can be interpreted as a quantification of the notion that higher dimensional difference sets can not lie too much in a lower dimensional subspace. We remark that this is sharp as well, up to an $O_{d}(1)$ error term. In order to see this, we return to the sets $A_N$ from $\eqref{trin1}$: each $A_N$ is supported on $d$ translates of $L_{e_d}$ and satisfies the inequality
\[ |(A_N - A_N) \setminus L_{e_d}| \leq d(d-1) (2N -1) =(2d-2) |A_N| - O_{d}(1). \]
\par

In our proofs of Theorems $\ref{diff}$ and $\ref{diffln}$, we are often interested in analysing sumsets of the form $A+B$ where $A$ and $B$ are finite subsets of $\mathbb{R}^d$ such that $\dim(A+B) = d$ and $|A| \geq |B|$. A classical result of Ruzsa \cite{Ru1994} examines such cases and implies that 
\begin{equation} \label{rt21}
 |A+B| \geq |A| + d |B|  - d(d+1)/2.
 \end{equation}
A key ingredient in our method will be a more refined version of this inequality, where we are interested in obtaining a stronger multiplicative factor for $|B|$ in the above bound, conditional on some structural information about $A$ and $B$. In order to present this, we record some further notation: given $\vec{v} \in \mathbb{R}^d \setminus \{0\}$, we write $H_{\vec{v}}$ to be the hyperplane orthogonal to $\vec{v}$, and we denote $\pi_{\vec{v}} : \mathbb{R}^d \to H_{\vec{v}}$ to be the natural projection map.

\begin{theorem} \label{mn2}
Let $r_1, r_2, d\in \mathbb{N}$ satisfy $r_1 \geq d\geq 2$ and let $A,B$ be finite, non-empty subsets of $\mathbb{R}^d$ such that $|A| \geq |B|$ and $\dim(A+B) = d$. Furthermore, let $\vec{v} \in \mathbb{R}^d \setminus \{0\}$ satisfy $|\pi_{\vec{v}}(A)| = r_1$ and $|\pi_{\vec{v}}(B)| = r_2$. Then
\[ |A+B| \geq |A| + \bigg(d+1 - \frac{1}{r_1 -d +2} - \frac{1}{r_2 - c + 2}\bigg)|B| - (d-1)(r_1 + r_2), \]
where $c = d$ whenever $r_2 \geq d$, and $c = \dim(B)$ whenever $r_2 < d$. 
\end{theorem}

The reader may note that since $r_1 \geq d$ and $r_2 \geq c$, this estimate recovers $\eqref{rt21}$ up to an $O_{d}(r_1 + r_2)$ term. For our purposes, we will be more interested in the cases when $r_1 > d$ and $r_2 > c$ while $r_1 + r_2 = O_d( |A|^{1- \delta})$ for some $\delta >0$. In such a setting, we see that the doubling factor attached to $|B|$ in the conclusion of Theorem $\ref{mn2}$ exceeds $d$ by some absolute positive constant. In fact, Theorem $\ref{mn2}$ can be used to furnish various types of structure theorems for sumsets of the form $A+B$. For instance, note that whenever $r_1 > d$ or $r_2 > c$, we get
\[  \bigg(d+1 - \frac{1}{r_1 -d +2} - \frac{1}{r_2 - c + 2}\bigg) \geq d+\frac{1}{6}, \]
which we may then combine with Theorem $\ref{mn2}$ to deduce the following result.

\begin{Corollary} \label{rnd2}
Let $A, B \subseteq \mathbb{R}^d$ be non-empty sets satisfying $\dim(A) = d$ and $|A| \geq |B|$ and
\begin{equation} \label{hyp8}
|A+B| \leq |A| + (d + 1/7)|B| - O_d(1). 
\end{equation}
Then there exists $\vec{v} \in \mathbb{R}^d \setminus \{0\}$ such that either $|\pi_{\vec{v}}(A)| = d$ and $|\pi_{\vec{v}}(B)| \in \{d, \dim B\}$, or $|B| = O_{d}(|\pi_{\vec{v}}(A)|  + |\pi_{\vec{v}}(B)|) $. 
\end{Corollary}

Indeed, we can generate many more such structure theorems by choosing different ranges for $r_1 - d$ and $r_2 - c$ in Theorem $\ref{mn2}$. This would lead to larger doubling factors being permissible in the hypothesis $\eqref{hyp8}$, but at the same time we would obtain weaker structural information about the sets $A$ and $B$. Here, we refer to bounds on $|\pi_{\vec{v}}(A)|$ and $|\pi_{\vec{v}}(B)|$ as structural information, since these tell us about the distribution of the sets $A$ and $B$ over translates of a $1$-dimensional subspace $L_{\vec{v}}$. 
\par

Our method also provides more specific structural results for difference sets, that is, we are able to characterise $d$-dimensional sets $A$ of $\mathbb{R}^d$ for which $|A-A|$ is close to the optimal lower bounds presented in Theorem $\ref{diff}$. In particular, we are able to show that such sets $A$ can be covered by a union of $2d-2$ parallel lines, up to some $O_{d}(|A|^{1- \delta})$ elements, where these lines themselves are contained in two translates of some hyperplane. We point the reader to Theorem $\ref{dbdg}$ in \S4 for further details regarding this.
\par

We now present a brief outline of our paper. We use \S2 to collect some preliminary lemmas from additive combinatorics. Next, we employ \S3 and \S4 to prove Theorems $\ref{diff}$ and $\ref{diffln}$ under the assumption that Theorem $\ref{mn2}$ holds true. The rest of the paper is then dedicated to proving Theorem $\ref{mn2}$. We begin by recording some results surrounding the technique of compressions in \S5. We divide the proof of Theorem $\ref{mn2}$ into two cases, depending on whether $r_2 \geq d$ or $r_2 < d$. The former is treated in \S6 in the form of Theorem $\ref{rszgn}$ and the latter is analysed in \S7 as Theorem $\ref{rszgn2}$.

\section{Additive combinatorics preliminaries}

In this section, we gather some preliminary results from additive combinatorics that we will employ throughout our paper, beginning with a classical lower bound for sumsets in Euclidean spaces, namely that if $d \in \mathbb{N}$ and $A, B$ are finite non-empty subsets of $\mathbb{R}^d$, then
\begin{equation} \label{runit}
 |A+B| \geq |A| + |B| - 1. 
 \end{equation}
 While this is sharp in general, one can deduce much more by analysing the dimension of the sets $A$ and $B$, and this is precisely the content of the aforementioned result of Ruzsa \cite{Ru1994}.

 \begin{lemma} \label{1ruzsa}
Given $d \in \mathbb{N}$ and finite non-empty subsets $A, B$ of $\mathbb{R}^d$ such that $\dim(A+B) =d$ and $|A| \geq |B|$, then 
\begin{equation} \label{ruzsa4}
 |A+B| \geq |A| + d |B| - d(d+1)/2. 
 \end{equation}
 \end{lemma}
 
This already implies that $|A-A| \geq (d+1)|A| - O_{d}(1)$ for all $d$-dimensional subsets of $\mathbb{R}^d$. As previously noted, Stanchescu \cite{St1998} improved upon this lower bound in the case when $d=3$ and we will use this result as a base case in our proofs.
 
\begin{lemma} \label{stanc}
Let $A \subseteq \mathbb{R}^3$ be a finite set such that $\dim(A) = 3$. Then 
\[ |A-A| \geq (4 + 1/2) |A| - O(1). \]
\end{lemma}

We shall also use the following structure theorem from \cite{Ak2019}, which allows us to cover sets having small sumsets with a small number of translates of a line.

\begin{lemma} \label{fri}
Let $A$ be a finite subset of $\mathbb{R}^d$ with $|A| = n$, where $n$ is large enough. If
\begin{equation} \label{ti} |A+A| \leq c^6n, \end{equation}
for some $c > 0$, then there exist parallel lines $l_1, l_2, \dots, l_r$ in $\mathbb{R}^d$, and constants $0 < \sigma \leq 1/2$ and $C_1 > 0$ depending only on $c$ such that 
\[ |A \cap l_1| \geq \dots \geq  |A \cap l_r| \geq  |A \cap l_1|^{1/2} \geq C_1^{-1} n^{\sigma}. \]
and 
\[ |A\setminus (l_1 \cup l_2 \cup \dots \cup l_r)| < C_1 c^6 n^{1-\sigma}. \] 

\end{lemma}

We now turn to other inverse type results for sumsets. We will use the following result of Grynkiewicz and Serra \cite[Theorem 1.3]{GS2010}.

\begin{lemma} \label{gsne2}
Let $A, B \subseteq \mathbb{R}^2$ be finite, non-empty sets and let $\vec{v} \in \mathbb{R}^2 \setminus \{0\}$ satisfy $|\pi_{\vec{v}}(A)| = r_1$ and $|\pi_{\vec{v}}(B)| = r_2$. Then
\begin{equation}  \label{gs1}
 |A+B| \geq \bigg( \frac{|A|}{r_1} + \frac{|B|}{r_2} - 1 \bigg)  (r_1 + r_2 - 1). 
\end{equation}
\end{lemma}

Finally, we will use a standard inequality to move from difference sets to sumsets, and we mention this as it is stated in \cite[Corollary 2.12]{TV2006}.

\begin{lemma} \label{rusza}
Suppose that $U,V$ are finite sets in some abelian group $G$. Then 
\begin{equation} \label{ru2} |U+V| \leq \frac{|U-V|^3}{|U||V|}.  \end{equation}
\end{lemma}

\section{Proofs of Theorems $\ref{diff}$ and $\ref{diffln}$}

We begin by following our previous approach in \cite{Ak2021}. That is, we first reduce to the case when $A$ is contained in few translates of a $1$-dimensional subspace. So let $A \subseteq \mathbb{R}^d$ be a finite, non-empty set such that $\dim(A) = d$ and $|A| = n$ where $n$ is a large enough natural number. We may assume that $|A-A| \leq 8(d-1)|A|$, since otherwise we are done, and so, we apply $\eqref{ru2}$ with $U= V = A$ to show that
\[ |A+A| \leq |A-A|^3 |A|^{-2} \leq (8d-8)^3 |A|. \]
We now apply Lemma $\ref{fri}$ with $c = (8d-8)^{1/2}$ to get parallel lines $l_1, l_2, \dots, l_r$ in $\mathbb{R}^d$, and constants $0 < \sigma \leq 1/2$ and $C_1 > 0$ depending only on $d$ such that 
\begin{equation} \label{ld}
 |A \cap l_1| \geq \dots \geq  |A \cap l_r| \geq  |A \cap l_1|^{1/2} \geq C_1^{-1} n^{\sigma}. 
 \end{equation}
and 
\[ |A\setminus (l_1 \cup l_2 \cup \dots \cup l_r)| < C_1 c^{6} n^{1-\sigma}. \]
Writing $S = A \cap (l_1 \cup l_2 \cup \dots \cup l_r)$ and $E = A \setminus S$, we note that $\eqref{ld}$ gives us 
\[ |A| \geq |S| = \sum_{i=1}^{r} |A \cap l_i| \geq r C_1^{-1} |A|^{\sigma}, \]
which, in turn, implies that
\begin{equation} \label{ubr}
 r \leq C_1 |A|^{1-\sigma}.
 \end{equation}
\par

We will now show that it suffices to prove Theorem $\ref{diff}$ for the set $S$. Assuming that Theorem $\ref{diff}$ holds for the set $S$, we divide our proof into two cases. First, if $\dim(S) = d$, then we have that
\begin{align*}
|A-A|  &  \geq |S-S| \geq (2d-2  +  1/({d -1}))|S| - O_d(|S|^{1- \delta}) \\
& \geq (2d-2 + 1/({d -1}))(|A| - |E|) - O_d(|A|^{1-\delta}) \\
& \geq (2d-2  +  1/({d -1}))|A| - O_d(|A|^{1 - \sigma}) + O_d(|A|^{1-\delta}) \\
& = (2d-2  +  1/({d -1}))|A| - O_d(|A|^{1- \min{(\sigma, \delta)}}).
\end{align*}
Since both $\delta$ and $\sigma$ are strictly positive constants that only depend on $d$, we see that $\min{(\sigma, \delta)}$ is also a strictly positive constant depending only on $d$, and consequently, our claim is verified when $\dim(S) =d$.
\par

Our second case is when $\dim(S) = d_1 < d$, in which case there are linearly independent elements $a_1, \dots, a_{d-d_1} \in E$ such that $\dim(S \cup \{a_1, \dots, a_{d-d_1} \}) = d$. This also implies that $a_1, \dots, a_{d-d_1}$ lie outside the affine span of $S$, and so, we have that the sets $S - S, S- a_1, \dots, S-a_{d-d_1}, a_1 - S, \dots, a_{d-d_1} - S$ are pairwise disjoint. Consequently, we infer that
\begin{align*}
 |A-A|  &  \geq |S-S| + \sum_{i=1}^{d-d_1} (|S- a_i| + |a_i - S|) \\
 & \geq (2d_1-2  + 1/({d_1-1}))|S|  - O_d(|S|^{1- \delta})  + \sum_{i=1}^{d-d_1} 2|S|\\
 & \geq (2d - 2 +  1/({d -1})) |S| - O_d(|S|^{1-\delta}) \\
 & \geq (2d-2 +  1/({d -1}))|A| -  O_d(|A|^{1- \min{(\sigma, \delta)}}).
\end{align*}
\par
As before, we see that $\min{(\sigma, \delta)}$ is a strictly positive constant depending only on $d$, and hence, our claim is proved. Hence, we will now prove that Theorems $\ref{diff}$ and $\ref{diffln}$ hold true for sets contained in a union of parallel lines.

\begin{Proposition} \label{3lines}
Let $d$ be a natural number and let $l_1, l_2, \dots, l_r$ be $r$ parallel lines in $\mathbb{R}^d$ parallel to $L_{\vec{v}}$ for some $\vec{v} \in \mathbb{R}^d \setminus \{0\}$. Suppose that $A \subseteq \mathbb{R}^d$ is a finite, non-empty set such that $A \subseteq l_1 \cup l_2 \cup \dots \cup l_r$ and $\dim(A) = d$. Then we have
\begin{equation} \label{out5}
 |(A-A) \setminus L_{\vec{v}}| \geq (2d-2)|A| - 2d^2 r.
 \end{equation}
Moreover, if we have that $d \geq 4$ and that $|A \cap l_i| \geq 2d^2$ for each $1 \leq i \leq r$, then
\begin{equation} \label{out4}
|A-A| \geq \bigg(2d-2 + \frac{1}{d-1}\bigg)|A| - K_d r,
 \end{equation}
where $K_d = 1000d^3$.
\end{Proposition}

We remark that Theorem $\ref{diff}$ follows from combining the preceding discussion with $\eqref{ld}$, $\eqref{ubr}$ and $\eqref{out4}$, while Theorem $\ref{diffln}$ follows from $\eqref{out5}$.
\par

We now begin the proof of Proposition $\ref{3lines}$. Our strategy will be to follow induction on the dimension $d$ and number of parallel lines $r$ that contain $A$.  Let $P(d,r)$ be the statement that $\eqref{out4}$ holds for $d$-dimensional sets $A$ which can be covered by $r$ parallel lines and let $Q(d,r)$ be the statement that $\eqref{out5}$ holds for $d$-dimensional sets $A$ which can be covered by $r$ parallel lines. We note that as $\dim(A) = d$, the number of parallel lines $r$ containing $A$ must always be at least $d$, and so, our base cases for $\eqref{out4}$ will be to prove $P(2,r)$ and $P(3,r)$ for all $r \geq 3$ and $P(d,d)$ for all $d \geq 4$. Similarly, for $\eqref{out5}$, our base case will be to prove $Q(1,r)$ holds true for each $r \in \mathbb{N}$ as well as that $Q(d,d)$ holds true for all $d \geq 2$, though the former can be noted to be trivially true.
\par

In our inductive step, for a given $d, r \in \mathbb{N}$ such that $r > d$ and $d \geq 2$, we will show that $Q(d,r)$ holds true if $Q(k,r-1)$ holds for all $k \leq d$. Similarly, we will show that whenever $r > d \geq 4$, we will prove that $P(d,r)$ holds if $P(k,r-1)$ holds for all $k \leq d$ as well as if $Q(d,r)$ holds true for all $r\geq d$. We will, hence, conclude that $Q(d,r)$ holds for all $r \geq d \geq 1$ and that $P(d,r)$ holds true for all $r \geq d \geq 3$.
\par

We begin with the base cases for $P(d,r)$. The case when $d \in \{2,3\}$ is handled by $\eqref{inm}$ and Lemma $\ref{stanc}$, and consequently, we proceed with our second base case, that is, when $r=d$. After an appropriate linear transformation, we may assume that $l_1$ is parallel to the vector $e_d$. Thus, writing $L_d = L_{e_d}= \{ \lambda \cdot e_d \ | \ \lambda \in \mathbb{R}\}$, we see that all the sets of the form $(A \cap l_i) - (A \cap l_j)$ are pairwise disjoint whenever $i \neq j$, as well as that they do not intersect $L_d $. Consequently, we have
\begin{align*}
|(A-A)\setminus L_d | & \ \geq \ \sum_{i \neq j} |(A \cap l_i) - (A \cap l_j)| \ \geq \ \sum_{i \neq j} (|A \cap l_i| + |A \cap l_j| - 1) \\
        & \ \geq 2(d-1) \sum_{i=1}^{d} |A \cap l_i|  - d^2 \ = \  2(d-1) |A| - d^2 .
\end{align*}
On the other hand, suppose that $|A \cap l_1| \geq |A \cap l_i|$ for each $1 \leq i \leq d$. Then, we have that
\[ |(A-A) \cap L_d | \geq |(A \cap l_1) - (A \cap l_1)|  \geq 2 |A \cap l_1| -1 \geq 2|A|/d -1. \]
Combining the above two estimates with the fact that $2/d \geq 1/(d-1)$ whenever $d \geq 4$, we see that we are done with our second base case for $P(d,r)$. 
\par

We further note that this argument also covers the base case $Q(d,d)$, so it suffices to check the validity of $Q(1,r)$ for all $r \in \mathbb{N}$, but this can be noted to be trivially true. Thus, we now move to the inductive step, which will be our primary focus in the next section.

\section{The Inductive Step}

Let $r,d$ be natural numbers such that $r > d \geq 2$. As previously mentioned, we assume that $P(k,r-1)$ holds for all $2 \leq k \leq d$. Furthermore, upon applying a suitable linear transformation, we may assume that $\vec{v} = e_d$. Let $H$ be the hyperplane that is orthogonal to $l_1$. For each $1 \leq i \leq r$, we write $x_i$ to be the point where $H$ and $l_i$ intersect, and we let $X = \{ x_1, \dots, x_r \}$. As $\dim(A) = d$, we see that $\dim(X) = d-1.$ Moreover, we denote $\pi$ to be the projection map from $\mathbb{R}^d$ to $H$. For any $Y \subseteq H$, we let $Y^{\pi}$ be a subset of $\mathbb{R}^d$ such that
\[Y^{\pi} = \{ x \in A \ | \ \pi(x) \in Y \}. \]
Thus $Y^{\pi}$ is the pre-image of $Y$ under $\pi$ in $A$. Because we are projecting along the direction of $l_1$ and $|A \cap l_i| \geq 2$ for all $1 \leq i \leq r$, we have $\dim(Y^{\pi}) = \dim(Y) + 1,$ for all $Y \subseteq \pi(A) := \{ \pi(a) \ | \ a \in A \} $. 
\par

We will use ${\|.\|}_d$ to denote the Euclidean norm in $\mathbb{R}^d$. As $H$ is a $(d-1)$-dimensional subspace of $\mathbb{R}^d$, we can find an invertible linear map $\phi$ from $H$ to $\mathbb{R}^{d-1}$. Fixing such a $\phi$, we can induce a norm ${\|.\|}_H$ on $H$ by writing ${\|x\|}_H = {\|\phi(x) \|}_{d-1}$, for all $x \in H$. 
\par

We now consider the convex hull $C$ of $X$. As $\dim(C) = \dim(X) = d-1$, we write $D_1, \dots, D_{t}$ to be the $(d-2)$-dimensional facets of $C$, where $t \in \mathbb{N}$ is suitably chosen. Furthermore, for any $i \neq j$, we observe that whenever the set $D_i \cap D_j$ is non-empty, it is contained in a $(d-3)$-dimensional affine subspace. Without loss of generality, we may assume that
\begin{equation} \label{34los}
| D_1^{\pi} | \geq |D_i^{\pi}| 
\end{equation}
for each $1 \leq i \leq t$. Let $H_1$ be the affine span of $D_1$ and let $l'$ be the line in $H$ that is orthogonal to $H_1$. 
\par

We cover $X$ with translates of $H_1$ and denote $H_2$ to be the translate of $H_1$ such that $H_2 \cap X \neq \emptyset$ and ${\| (l' \cap H_2) - (l' \cap H_1) \|}_{H}$ is maximised. The existence and uniqueness of such an $H_2$ is confirmed by the fact that $H_1$ is a $(d-2)$-dimensional subspace of $H$, that contains a $(d-2)$-dimensional facet of $C$ and $l'$ is chosen to be orthogonal to $H_1$. 
Thus for all translates $H'$ of $H_1$ such that $H' \cap X \neq \emptyset$ and $H' \neq H_2$, we have
\begin{equation}  \label{far}
{\| (l' \cap H_2) - (l' \cap H_1) \|}_{H} > {\| (l' \cap H') - (l' \cap H_1) \|}_{H}.
\end{equation}
Here, for ease of notation, we write $X_i = H_i \cap X$ and $Y_i = X \setminus X_i$ for $i=1,2$. We note that by definition of $X_2$, the set $X_2$ must lie in $D_i$ for some $2 \leq i \leq t$, which, in turn, combines with $\eqref{34los}$ to deliver the bound
\begin{equation} \label{36los}
|X_1^{\pi}| \geq |X_2^{\pi}|. 
\end{equation}
\par

Our strategy, now, essentially involves analysing how $X_1$, $X_2, Y_1$ and $Y_2$ interact with each other. We begin by translating our set $A$, and thus $X$, so that $0 \in H_2$. From $\eqref{far}$, we deduce that $X_2 - X_1, X_1 - X_2$ and $Y_2 - Y_2$ are pairwise disjoint and consequently, the sets $X_2^{\pi} - X_1^{\pi}, X_1^{\pi} - X_2^{\pi}$ and $Y_2^{\pi} - Y_2^{\pi}$ are pairwise disjoint. 
\par

We will now show that $Q(d,r)$ holds true if $Q(k,r-1)$ holds true for each $k \leq d$, and in fact, we will further divide our argument into two cases, the first being when $Y_2 \neq X_1$, that is, when $A$ is contained in more than two translates of the hyperplane $H_1$. In this case, we have $\dim(Y_2) = d-1$, which in turn implies that $\dim(Y_2^{\pi}) = d$. Moreover, we recall that $\dim(X_1^{\pi}) = \dim(X_1) + 1 = d-1$ and so, we apply Lemma $\ref{1ruzsa}$, together with $\eqref{36los}$, to see that
\begin{equation} \label{ind1}
 |X_1^{\pi} - X_2^{\pi}| \geq |X_1^{\pi}|  + (d-1)|X_2^{\pi}|  - d^2. 
 \end{equation}
Thus, upon amalgamating $Q(d,r-1)$ along with the fact that the sets $X_1^{\pi} - X_2^{\pi}, X_2^{\pi} - X_1^{\pi} $ and $L_{d}$ are pairwise disjoint, we find that
 \begin{align*}
 |(A-A) \setminus L_d| & \geq 2|X_1^{\pi} - X_2^{\pi}| + |(Y_2^{\pi} - Y_2^{\pi})\setminus L_d| \\
 & \geq 2(d-1)|X_2^{\pi}| + 2(d-1)|Y_2^{\pi}| - 2d^2 - 2d^2(r-1)  \\
 & \geq 2(d-1) |A| - 2d^2 r,
 \end{align*}
and so, we obtain $Q(d,r)$ in this case.
\par

On the other hand, when $Y_2 = X_1$, that is, when $\dim(Y_2^{\pi}) = d-1$, we may utilise $\eqref{ind1}$ and the hypothesis $Q(d-1,r-1)$ to deduce that
 \begin{align*}
 |(A-A) \setminus L_d| & \geq 2|Y_2^{\pi} - X_2^{\pi}| + |(Y_2^{\pi} - Y_2^{\pi})\setminus L_d| \\
 & \geq 2(d-1)|X_2^{\pi}| + 2|Y_2^{\pi}| + 2(d-2)|Y_2^{\pi}| - 2d^2 - 2d^2(r-1)  \\
 & \geq 2(d-1) |A| - 4d^2r.
 \end{align*}
This finishes our inductive step for $Q(d,r)$, whence, we have shown that $\eqref{out5}$ holds for all $r \geq d \geq 1$.
\par
 
We now proceed to show that given $r > d \geq 4$, proposition $P(d,r)$ holds true if $P(k,r-1)$ holds for all $2 \leq k \leq d$ as well as if $Q(d,r)$ holds true for all $r\geq d$. We begin by claiming that it suffices to consider the case when $Y_2 = X_1$. In order to see this, suppose that $Y_2 \neq X_1$, in which case, we have that $\dim(Y_2) = d-1$ and consequently, $\dim(Y_2^{\pi}) = d$. We may now combine our arguments from before with $P(d,r-1)$ to show that
\begin{align*}
|A-A| & \geq  | X_1^{\pi} - X_2^{\pi}| + | X_2^{\pi} - X_1^{\pi}| + |Y_2^{\pi} - Y_2^{\pi}| \\
        & \geq \bigg(2d - 2 + \frac{1}{d-1}\bigg)|Y_2^{\pi}| - K_{d}(r-1) + 2|X_1^{\pi}|  + 2(d-1)|X_2^{\pi}|  - 2d^2 \\
        & \geq \bigg(2d - 2 + \frac{1}{d-1}\bigg)|A|  + 3|X_1^{\pi}|/2 - K_d r,
\end{align*} 
which is even stronger than the desired bound.
\par

Hence, we now suppose that $Y_2 = X_1$, which allows us to infer that $\dim(Y_2) = d-2$ and consequently, $\dim(Y_2^{\pi}) = d-1$. Thus, by $P(d-1,r-1)$, we have
\begin{equation} \label{ca1}
 |Y_2^{\pi} - Y_2^{\pi}| \geq (2d - 4 + 1/{(d-2)})|Y_2^{\pi}| - K_{d-1} (r-1).
 \end{equation}
We are now interested in estimating $|Y_2^{\pi} - X_2^{\pi}|$. We may translate $X_2^{\pi}$ appropriately to ensure that $X_2^{\pi}$ is contained in the affine span of $Y_2^{\pi}$, whenceforth, applying Theorem $\ref{mn2}$ yields the bound
\begin{equation} \label{urgh23}
 |Y_2^{\pi} - X_2^{\pi}| \geq |Y_2^{\pi}|  + \bigg(d - \frac{1}{r_1 -(d-1) +2} - \frac{1}{r_2 - c + 2}\bigg)|X_2^{\pi}|  - (d-2)(r_1 + r_2), 
 \end{equation}
where $r_1 = \pi_{d}(Y_2^{\pi})$ and $r_2 = \pi_{d}(X_2^{\pi})$ and $c = d-1$ if $r_2 \geq d-1$ and $c=\dim(X_2^{\pi})$ if $r_2 < d-1$. Moreover, we have that $r= r_1 + r_2$. Note that $\eqref{urgh23}$ implies that if $r_1 > d-1$ or $r_2 > c$, then we obtain the bound
\[ |Y_2^{\pi} - X_2^{\pi}| \geq |Y_2^{\pi}|  +  (d-1 + 1/6  )|X_2^{\pi}|  - (d-2)(r_1 + r_2) , \]
which upon combining with $\eqref{ca1}$ gives us
\begin{align*}
|A-A| & \geq  |Y_2^{\pi} - Y_2^{\pi}|  + 2 |Y_2^{\pi} - X_2^{\pi}| \\
& \geq \bigg(2d - 2 + \frac{1}{d-2}\bigg)|Y_2^{\pi}| + \bigg(2d -2 + \frac{1}{3} \bigg)|X_2^{\pi}|  - (d-1)r - K_d (r-1) \\
& \geq \bigg(2d - 2 + \frac{1}{d-1}\bigg)|A| + \frac{|Y_2^{\pi}|}{(d-1)(d-2)} - K_d r \\
& \geq  \bigg(2d - 2 + \frac{1}{d-1}\bigg)|A| + \frac{|A|}{2(d-1)(d-2)} - K_d r .
\end{align*}
Thus, it suffices to investigate the cases when $r_1 = d-1$ and $r_2 = c$, that is, when $r = r_1 + r_2 \leq 2d-2$. 
\par

In the setting when $r \leq 2d-2$, we wield $Q(d,r)$ to deliver the bound
\[ |(A-A) \setminus L_d| \geq (2d-2) |A| - 2d^2 r. \]
Moreover,  we may suppose that, without loss of generality, we have $|l_1 \cap A| \geq |l_i \cap A|$ for each $1 \leq i \leq r$, in which case, we get that
\[ |(A-A) \cap L_d| \geq |(A \cap l_1) - (A \cap l_1)| \geq 2 |A \cap l_1| -1 \geq \frac{2}{r} |A| - 1. \]
Thus, when $r \leq 2(d-1)$, we may combine the above inequality with the preceding lower bound to get that
\begin{align*}
 |A-A| & \geq (2d - 2+ 2/r) |A| - 2d^2r - 1 \\
 & \geq \bigg(2d - 2 + \frac{1}{d-1}\bigg) |A| + \frac{2d-2 - r}{r(d-1)} |A| - K_d r . 
 \end{align*}
Inserting the bound $r \leq 2d-2$ delivers the desired estimate and consequently, we finish our proof of Proposition $\ref{3lines}$. 
\par

In fact, it is worth noting that we prove something stronger than just Theorem $\ref{diff}$ through the above argument, that is, we actually show that $|A-A| \geq (2d - 2 + 1/(d-1) + c_d) |A| - O_d(|A|^{1 - \delta})$ for some small constant $c_d >0$, unless $A$ satisfies some very specific combinatorial properties. We record this in the form of the following theorem. 

\begin{theorem} \label{dbdg}
Let $d \geq 4$ and let $A$ be a finite, non-empty subset of $\mathbb{R}^d$ such that $\dim(A) = d$. Then there exists some absolute constant $\delta >0$ that only depends on $d$ such that either 
\[ |A-A| \geq \bigg(2d - 2  + \frac{1}{d-1} + \frac{1}{(2d-3)(d-1)} \bigg) |A| - O_d(|A|^{1 - \delta}), \]
 or $A$ may be partitioned as $A= A_1 \cup A_2 \cup E$, where $A_1$ and $A_2$ lie in translates of some hyperplane with $\dim(A_1) = \dim(A_2) = d-1$ and $|\pi_{\vec{v}}(A_1)| = |\pi_{\vec{v}}(A_2)| = d-1$ for some $\vec{v} \in \mathbb{R}^d \setminus \{0\}$, and $|E| = O_{d}(|A|^{1- \delta})$.
\end{theorem}

Theorem $\ref{dbdg}$ can be interpreted as an inverse theorem for difference sets since it states that any set that has its difference set close to being optimally small satisfies very strict structural conditions, and in particular, apart from $O_{d}(|A|^{1- \delta})$ elements, any such set $A$ must lie on $2d-2$ translates of a $1$-dimensional subspace, which themselves are contained in two translates of a hyperplane.

\section{Linear transformations and Compressions}

For ease of notation, we will use $\pi_{i}$ to denote the projection map $\pi_{e_{i}}: \mathbb{R}^d \to H_{e_{i}}$ for every $1 \leq i \leq d$. We begin by presenting a lemma that implies that for the purposes of proving Theorem $\ref{mn2}$, it suffices to assume that $A$ and $B$ are finite subsets of $\mathbb{Z}^d$ that satisfy suitable properties.

\begin{lemma} \label{realin}
Let $d, r_1, r_2 \in \mathbb{N}$ such that $r_1 \geq d\geq 2$, let $A, B \subset \mathbb{R}^d$ and let $\vec{v} \in \mathbb{R}^d \setminus \{0\}$ satisfy $|\pi_{\vec{v}}(A)| = r_1$ and $| \pi_{\vec{v}}(B)| = r_2$ and $\dim(A+B) = d$ and $|A| \geq |B|$. Then either 
\begin{equation} \label{ur3}
 |A+B| \geq |A| + (d+1)|B| - (d-1)(r_1 + r_2)  ,
 \end{equation}
or there exist sets $A_1,B_1 \subseteq \mathbb{Z}^d$ satisfying the following. 
\begin{enumerate}
\item We have $|A_1| = |A|$ and $|B_1| = |B|$ and $|A+B| = |A_1+B_1|$ and $\dim(B) \geq 1$.
\item Writing $k = \dim(B)$, there exists some natural number $m$ such that $m \cdot \{0, e_1, \dots, e_{d-k} \} \subseteq A_1$ and  $m \cdot \{0, e_{d-k+1}, \dots, e_d\} \subseteq B_1$.
\item We have $|\pi_{d}(A_1)| = r_1$ and $|\pi_{d}(B_1)| = r_2$.
\end{enumerate}
\end{lemma}

\begin{proof}
 We begin by translating $A$ and $B$ to ensure that $0 \in A \cap B$, and since $\dim(A+B) = d$, and furthermore, we may assume that $r_2 \leq |B| -1$, since otherwise, we can use $\eqref{ruzsa4}$ to deduce $\eqref{ur3}$. This implies the existence of $\vec{b}_1, \vec{b}_2 \in B$ and $\lambda \in \mathbb{R} \setminus \{0\}$ that satisfy $\vec{b}_2 - \vec{b}_1 = \lambda  \cdot \vec{v}$. Since our hypothesis remains invariant under dilations of $\vec{v}$ and translations of $B$, we may choose $\vec{b}_1 = 0$ and $\lambda = 1$, whereupon, we get that $\vec{v} \in B$.
\par

Next, we apply an appropriately chosen invertible linear transformation to ensure that
\begin{equation} \label{gsf}
 \vec{v} = e_d \ \text{and} \ \{0, e_1, \dots, e_{d-k}\} \subseteq A \ \text{and} \ \{0, e_{d-k+1}, \dots, e_{d} \} \subseteq B. 
 \end{equation}
We now write
\[ \mathcal{S} = \big\{ \sum_{i=1}^{n} r_i \cdot \vec{x}_i \ | \ n \in \mathbb{N}, \ \vec{x}_1, \dots, \vec{x}_n \in A+B , \ r_1, \dots, r_n \in \mathbb{Q}\big\}\]
to be the span of $A+B$ over the field $\mathbb{Q}$, and we denote $\{\vec{u}_1, \dots, \vec{u}_n\}$ to be the basis of $\mathcal{S}$ over $\mathbb{Q}$. Noting $\eqref{gsf}$, we see that $\{e_1, \dots, e_d\} \subseteq \{\vec{u}_1, \dots, \vec{u}_n\}$, and so, we may assume that $\vec{u}_i = e_i$ for $1 \leq i \leq d$. We define the linear and injective map $\psi : \mathcal{S} \to \mathbb{Q}^n$ to be 
\[ \psi(r_1 \cdot \vec{u}_1 + \dots + r_n \cdot \vec{u}_n) = m \cdot (r_1, \dots, r_n), \]
where $m$ is a suitably chosen natural number such that $\psi(A+B) \subseteq \mathbb{Z}^n$. Since $\psi$ is linear and injective, we have $|A| = |\psi(A)|$ and $|B| = |\psi(B)|$ and $|A+B| = |\psi(A+B)|$, as well as that $m \cdot \{0, e_1, \dots, e_{d-k} \} \subseteq \psi(A) \ \text{and} \ m \cdot \{0, e_{d-k+1}, \dots, e_{d} \} \subseteq \psi(B).$ Moreover, we see that $\dim(\psi(A+B)) = n$, and so, if $n> d$, we can then apply $\eqref{ruzsa4}$ to deduce $\eqref{ur3}$. This means that it suffices to consider the case when $n = d$.
\par

Finally, for any $\vec{a} \in A$, let $\vec{u} \in H_{\vec{v}}$ and $t \in \mathbb{R}$ satisfy $\vec{a} = \vec{u} + t \cdot \vec{v}$. This actually implies that $t = (\vec{a} \cdot \vec{v} )(\vec{v} \cdot \vec{v})^{-1} \in \mathbb{Q}$, which, in turn, gives us that $\vec{u} \in \mathcal{S}$, consequently allowing us to deduce that $\psi(\vec{a}) = t \cdot \psi(\vec{v}) + \psi(\vec{u})$. Putting this together with the fact that $\psi$ is linear and injective on $\mathcal{S}$ delivers the estimate $|\pi_{\psi(\vec{v})}(\psi(A))| = r_1$. One may proceed similarly to show that $|\pi_{\psi(\vec{v})}(\psi(B))| = r_2$, whereupon, we conclude our proof by noting that $\psi(\vec{v}) = \psi(e_d) = e_d$.
\end{proof}

We now present a brief discussion on the technique of compressions, and we do so by importing some notations and definitions from \cite{GG2001}. Writing $E = \mathbb{N} \cup \{0\}$, we define a set $A \subseteq E^d$ to be a \emph{down set} if for each $1 \leq i \leq d$ and for each $(a_1, \dots, a_d) \in A$, we have
\[ \{ (a_1, \dots, a_{i-1}, 0, a_{i+1}, \dots, a_d) + b \cdot e_i \ | \ 0 \leq b \leq a_i -1 \} \subseteq A. \]
In particular, this implies that $([0, a_1] \times \dots \times [0, a_d]) \cap \mathbb{Z}^d \subseteq A$, for each $(a_1, \dots, a_d) \in A$. Given a vector $\vec{v} \in \mathbb{R}^d \setminus \{0\}$, we define $L_{\vec{v}} = \{ \lambda \cdot \vec{v} \ | \ \lambda \in \mathbb{R} \}$.
\par

Next, we write 
\[ \mathcal{W} = \{ \vec{v} \in \mathbb{Z}^d \ | \ v_i = - 1 \ \text{for some} \ 1 \leq i \leq d \ \text{and} \ v_j \geq 0 \ \text{for} \ j \neq i, \ 1 \leq j \leq d \}. \]
If $\vec{v} \in \mathcal{W}$, let $\mathbb{Z}(\vec{v}) = \{ \vec{x} \in E^d \ | \ \vec{x} + \vec{v} \notin E^d \}$. Suppose that $A$ is a finite subset of $E^d$ and $\vec{x} \in \mathbb{Z}(\vec{v})$, then the $\vec{v}$-section of $A$ at $\vec{x}$ is 
\[ A_{\vec{v}}(\vec{x}) = \{ m \in \mathbb{N} \ | \ \vec{x} - m \cdot \vec{v} \in A\}. \]
Finally, we define the \emph{$\vec{v}$-compression} $\mathcal{C}_{\vec{v}}(A)$ of the set $A$ to be the unique set in $E^d$ satisfying
\[ (\mathcal{C}_{\vec{v}}(A))_{\vec{v}}(\vec{x}) = \{ 0, 1, \dots, |A_{\vec{v}}(\vec{x}) | - 1\}, \]
for each $\vec{x} \in \mathbb{Z}(\vec{v})$. The set $A$ is called $\vec{v}$-compressed if $\mathcal{C}_{\vec{v}}(A) = A$.
\par

We now record \cite[Corollary $3.5$]{GG2001} that implies that cardinalities of sumsets do not increase when the individual sets are compressed in suitable directions.

\begin{lemma} \label{gardc}
Let $A$ and $B$ be finite subsets of $E^d$ and let $\vec{v} \in \mathcal{W}$. Then we have
\[ |A+B| \geq |\mathcal{C}_{\vec{v}}(A) + \mathcal{C}_{\vec{v}}(B)|. \]
\end{lemma}

We also present the following lemma that records the fact that certain compressions preserve certain kind of structures.

\begin{lemma} \label{random}
Let $m \in \mathbb{N}$, let $A$ be a finite subset of $E^d$ and let $\vec{u} \in \{-e_1, \dots, -e_d\} \subseteq \mathcal{W}$, and let $\vec{v}_1, \dots, \vec{v}_m \in H_{\vec{u}} \cap \mathcal{W}$. Then
\begin{equation} \label{comp3}
  |\pi_{\vec{u}}(A)| =  |\pi_{\vec{u}}(\mathcal{C}_{\vec{u}} (A))| = |\pi_{\vec{u}}(  \mathcal{C}_{\vec{v}_m}( \dots  \mathcal{C}_{\vec{v}_1}(\mathcal{C}_{\vec{u}} (A)) \dots ))|.
  \end{equation}
Moreover, if $k, m \in \mathbb{N}$ and $\vec{t} \in E^d$ satisfy $m \cdot \{0, e_1, \dots, e_k\} + \vec{t} \subseteq A$, then we have
\begin{equation} \label{comp2}
\{0, e_1, \dots, e_k\} \subseteq \mathcal{C}_{-\vec{e}_k}( \dots  \mathcal{C}_{-\vec{e}_1}(A) \dots ). \end{equation}
\end{lemma}

\begin{proof}
By definition of $\mathcal{C}_{\vec{u}}$, it follows that $|\pi_{\vec{u}}(A)| = |\pi_{\vec{u}}(\mathcal{C}_{\vec{u}} (A))|$. Combining this with the facts that $\pi_{\vec{u}}(\mathcal{C}_{\vec{u}}(A)) = \mathcal{C}_{\vec{u}}(A) \cap H_{\vec{u}}$ and that for any $\vec{v} \in H_{\vec{u}} \cap \mathcal{W}$ and for any finite $X \subseteq H_{\vec{u}}$ we have $|\mathcal{C}_{\vec{v}}(X)| = |X|$, allows us to deduce $\eqref{comp3}$. One may similarly prove that $\eqref{comp2}$ holds, for instance, this is mentioned in the proof of \cite[Corollary $2.6$]{GG2001}.
\end{proof}

We will use Lemma $\ref{random}$ throughout the proofs of the next two lemmas, which themselves focus on reducing our proof to the case when $A$ and $B$ are down sets in $E^d$ satisfying suitable arithmetic properties. 

\begin{lemma} \label{cmp1}
Let $r_1, r_2, k, m \in \mathbb{N}$ and let $A$, $B$ be  finite subsets of $\mathbb{Z}^d$ such that $|\pi_{d}(A)| = r_1$ and $|\pi_{d}(B)| = r_2$ and $\dim(B) = k$. Moreover, suppose that
\[ m \cdot \{0, e_1, \dots, e_{d-k}\} \subseteq A \ \text{and} \ m \cdot \{0,e_{d-k+1}, \dots, e_d\} \subseteq B. \]
Then there exist down sets $A'$ and $B'$ satisfying the following conditions
\begin{enumerate}
 \item We have $\dim(A'+B') = d$ and $|A'| = |A|$ and $|B'| = |B|$, and $|A+B| \geq |A' + B'|.$ 
 \item Writing $U' = \pi_{d}(A')$ and $V' = \pi_{d}(B')$, we have that $|U'| = r_1$ and $|V'| = r_2$. 
 \item We either have $U' \subseteq \mathbb{Z}^{d-2} \times \{0\}$ and $V' \setminus (\mathbb{Z}^{d-2} \times \{0\}) = \{e_{d-1}\}$, or $V' \subseteq \mathbb{Z}^{d-2} \times \{0\}$ and $U' \setminus (\mathbb{Z}^{d-2} \times \{0\}) = \{e_{d-1}\}$, or $U' \setminus (\mathbb{Z}^{d-2} \times \{0\}) =  \{e_{d-1}\}$ and $U'= V'$.
 \end{enumerate}
\end{lemma}

\begin{proof} 
Let $\vec{t} \in E^d$ satisfy the condition that both $A + \vec{t}$ and $B+ \vec{t}$ are subsets of $E^d$, and set
\[ A_ 1= \mathcal{C}_{-e_{d-1}} ( \dots \mathcal{C}_{-e_1}(\mathcal{C}_{-e_d}(A+ \vec{t}))\dots) \ \text{and} \ B_1= \mathcal{C}_{-e_{d-1}}( \dots \mathcal{C}_{-e_1}(\mathcal{C}_{-e_d}(B+ \vec{t}))\dots). \]
Noting Lemma $\ref{random}$, we see that $|\pi_d(A_1)| = |\pi_d(A)|$ and $|\pi_d(B_1)| = |\pi_d(B)|$ as well as that $\dim(A_1 + B_1) = d$. Writing $U_1 = \pi_d(A_1)$ and let $V_1 = \pi_d(B_1)$, we see that $e_{d-1} \in U_1 \cup V_1$. If $e_{d-1} \in U_1$, define $X_1 = U_1 \cup (\mathbb{Z}^{d-1} \times \{0\}$, otherwise, define $X_1 = \emptyset$. Similarly, if $e_{d-1} \in V_1$, define $Y_1 = V_1 \cup (\mathbb{Z}^{d-1} \times \{0\})$, else, let $Y_1 = \emptyset$. Since $e_{d-1} \in U_1 \cup V_1$, we have that $X_1 \cup Y_1 \neq \emptyset$. Let $\vec{y} \in X_1 \cup Y_1$ be the element that maximises $\n{\vec{y} - e_{d-1}}_{d-1}$. Setting $\vec{w} = \vec{y} - e_{d-1}$, we compress $A_1$ and $B_1$ in the direction of $\vec{w}$, that is, we set 
\[A_2 = \mathcal{C}_{-e_{d-2}} ( \dots (\mathcal{C}_{-e_1}(\mathcal{C}_{\vec{w}}(A_1)))) \ \text{and} \ B_2 = \mathcal{C}_{-e_{d-2}} ( \dots (\mathcal{C}_{-e_1}(\mathcal{C}_{\vec{w}}(B_1)))). \]
Note that our choice of $X_1$ and $Y_1$ implies that $e_{d-1} \in A_2 \cup B_2$, and so, we may deduce that $\dim(A_2 + B_2) = d$, and furthermore, we have $|\pi_d(A_2)| = |\pi_d(A_1)|$ and $|\pi_d(B_2)| = |\pi_d(B_1)|$. Finally, set $A_3 = \mathcal{C}_{-e_d}(A_2)$ and $B_3 =\mathcal{C}_{-e_d}(B_2)$, and write $U_3 = \pi_d(A_3)$ and $V_3 = \pi_d(B_3)$.
 \par
 
Our choice of $\vec{w}$ implies that $U_3 \cup V_3  \subseteq \mathbb{Z}^{d-2} \times \{0,1\}$ as well as that $e_{d-1} \in U_3 \cup V_3$. Thus, defining $X_3$ and $Y_3$ analogously to $X_1$ and $Y_1$, we repeat the above procedure again by finding $\vec{y}' \in X_3 \cup Y_3$ such that $\n{\vec{y}' - e_{d-1}}_{d-1}$ is maximised. We now compress in the directions of $\vec{y}' - e_{d-1}, - e_1, \dots, -e_{d-2}, -e_d$ in that order to obtain sets $A_4$ and $B_4$. This time, writing $U_4 =\pi_{d}(A_4)$ and $V_4 = \pi_{d}(B_4)$, we see that our choice of $\vec{w}'$ indicates that $(U_4 \cup V_4) \setminus (\mathbb{Z}^{d-2} \times \{0\}) = \{e_{d-1}\}$.
\par

 If $e_{d-1} \in U_4 \setminus V_4$, then we have $V_4 \subseteq \mathbb{Z}^{d-2} \times \{0\}$, and so we are done. Similarly, if $e_{d-1} \in V_4 \setminus U_4$, we are done as well, so we may assume that $e_{d-1} \in V_4 \cap U_4$. Defining $W =  (U_4 \setminus V_4) \cup (V_4 \setminus U_4)$, we see that it suffices to consider the case when $W$ is non-empty, since otherwise we would have $U_4 = V_4$ and we would be done. Thus, let $\vec{x}$ be an element of $W$ such that $\n{\vec{x}}_{d-1}$ is maximal. In this case, we may write
$\vec{w}'' = \vec{x} - e_{d-1}$ and set
\[ A_5 = \mathcal{C}_{-e_d}(\mathcal{C}_{-e_{d-2}} ( \dots (\mathcal{C}_{-e_1}(\mathcal{C}_{\vec{w}''}(A_4))))) \ \text{and} \ B_5 = \mathcal{C}_{-e_d}(\mathcal{C}_{-e_{d-2}} ( \dots (\mathcal{C}_{-e_1}(\mathcal{C}_{\vec{w}''}(B_4))))). \]
Defining $U_5  = \pi_d(A_5)$ and $V_5 = \pi_d(B_5)$, we see that by our choice of $\vec{x}$, we must have either $U_5 \subseteq \mathbb{Z}^{d-2} \times \{0\}$ or $V_5 \subseteq \mathbb{Z}^{d-2} \times \{0\}$, and consequently, we finish the proof of our lemma. 
\end{proof}

\begin{lemma} \label{tmb}
Let $r_1, r_2, d, k \in \mathbb{N}$ satisfy $r_1 \geq d> r_2 \geq k$ and $d \geq 2$. Let $A,B$ be finite subsets of $\mathbb{Z}^d$ such that $|A| \geq |B| \geq 2$ and $\dim(A+B) = d$ and $\dim(B) = k$. Furthermore, let $\vec{v} \in B$ satisfy $|\pi_{\vec{v}}(A)| = r_1$ and $|\pi_{\vec{v}}(B)| = r_2$. Then there exist down sets $A',B'$ satisfying the following. 
\begin{enumerate}
 \item We have $|A'| = |A|$ and $|B'| = |B|$ and $\dim(B') = k$ and $|A+B| \geq |A'+B'|$. 
 \item We have $\{e_1, \dots, e_k\} \subseteq B' \subseteq \mathbb{Z}^k \times \{0\}^{d-k}$ and $A' \setminus (\mathbb{Z}^k \times \{0\}^{d-k}) = \{ e_{k+1}, \dots, e_{d}\}$.
  \item We have $|\pi_{1}(A')| = r_1$ and $|\pi_{1}(B')| = r_2$. 
 \end{enumerate}
\end{lemma}

\begin{proof}
We begin by translating $A$ and $B$ appropriately in order to ensure that $0 \in A \cap B$. Moreover, since $\vec{v} \in B$ and $\dim(B) = k$ and $\dim(A+B) = d$, there exists some non-singular linear transformation $\phi : \mathbb{Z}^d \to \mathbb{Z}^d$ and some natural number $m \in \mathbb{N}$ such that
\[ m \cdot e_1 = \phi(\vec{v}) \ \text{and} \ m \cdot \{0, e_1, \dots, e_k\} \subseteq \phi(B) \subseteq \mathbb{Z}^k \times \{0\}^{d-k} \ \text{and} \ m \cdot \{ 0,e_{k+1} , \dots, e_{d} \} \subseteq \phi(A). \]
Setting $A_1 = \mathcal{C}_{-e_d}(\dots ( \mathcal{C}_{-e_1}(A+ \vec{t})) \dots )$ and $B_1 =\mathcal{C}_{-e_d}(\dots ( \mathcal{C}_{-e_1}(B +\vec{t})) \dots )$, where $\vec{t} \in E^d$ satisfies the fact that $A+ \vec{t}, B+ \vec{t} \subseteq E^d$, we use Lemma $\ref{random}$ to deduce that
\[ \{0, e_1,\dots, e_k \} \subseteq B_1 \ \text{and} \ \{0, e_{k+1}, \dots, e_d \} \subseteq A_1. \]
\par

Consider $U_1 = A_1 \cap (\{0\} \times \mathbb{Z}^{d-2} \times \{0\})$ and let $\vec{w} \in U_1$ maximise $\n{e_{d} - \vec{w}}_{d}$, and so, we set $\vec{x} = e_d - \vec{w}$ and $A_2 = \mathcal{C}_{-e_{d}}(\dots  \mathcal{C}_{-e_{2}}(  \mathcal{C}_{-\vec{x}}( A_1 ) ) \dots )$ and $B_2 = \mathcal{C}_{-e_{d}}(\dots  \mathcal{C}_{-e_{2}}(  \mathcal{C}_{-\vec{x}}( B_1 ) ) \dots )$. Note that $B_2 = B_1$, while the set $A_2$ would satisfy 
\[ A_2 \cap (\{0\} \times \mathbb{Z}^{d-1}) \subseteq \{0\} \times \mathbb{Z}^{d-2} \times\{0,1\}, \]
since $A_1$ was a down set. We repeat this process again, by writing $U_2 = A_2 \cap (\{0\} \times \mathbb{Z}^{d-2} \times \{0\})$ and letting $\vec{w}' \in U_2$ maximise $\n{e_{d} - \vec{w}'}_{d}$ and compressing our sets in the directions of $-e_d + \vec{w}', -e_2, \dots, -e_d$ in that order. We would now obtain sets $A_3$ and $B_3$ such that $B_3 = B_2$ while the set $A_3$ would satisfy
\[ A_3 \cap (\{0\} \times \mathbb{Z}^{d-1}) \subseteq (\{0\} \times \mathbb{Z}^{d-2} \times\{0\}) \cup \{ e_d\}. \]
This would imply that $\pi_{1}(A_3) \setminus (\mathbb{Z}^{d-2} \times \{0\}) = \{ e_{d}\}$, while at the same time, we would also have $|\pi_1(A_3)| = r_1$ and $|\pi_1(B_3)| = r_2$. Now, we choose $\vec{w}'' \in (\mathbb{Z} \times \{0\}^{d-1}) \cap A_2$ such that $\n{\vec{w}''}_{d}$ is maximal, and we set $\vec{x}'' = e_{d} - \vec{w}''$ and $A_4 = \mathcal{C}_{-e_1}(\mathcal{C}_{-\vec{x}''}(A_3))$ and $B_4 = \mathcal{C}_{-e_1}(\mathcal{C}_{-\vec{x}''}(B_3))$. Note that $A_4$ and $B_4$ are down sets such that $B_4= B_1$ and $|\pi_1(A_4)|= r_1$ and $A_4 \setminus (\mathbb{Z}^{d-1} \times \{0\}) = \{e_{d}\}$. Iterating this argument $d-k-1$ times delivers the desired result.
\end{proof}

\section{Lower bounds for sumsets} 

For the sake of exposition, in this section we adopt the following notation. Given a finite, non-empty subset $A$ of $\mathbb{R}^d$ and a vector $\vec{x} \in \mathbb{R}^{d-1}$, we define
\[ \vec{x}^A = \{ \vec{a} \in A \ | \ \pi_d(\vec{a}) = \vec{x} \}. \]
We begin this section by proving a simple lower bound for $|A+B|$ when $B$ is contained in a $1$-dimensional affine subspace.

\begin{lemma} \label{d12}
Let $A,B$ be finite subsets of $\mathbb{R}^d$ and let $\vec{v} \in \mathbb{R}^d \setminus \{0\}$ satisfy $|\pi_{\vec{v}}(B)| =1$ and $|\pi_{\vec{v}}(A)| = r$. Then
\[ |A+ B| \geq |A| + r|B| - r. \]
\end{lemma}

\begin{proof}
We write $A = A_1 \cup \dots \cup A_{r}$ where $\pi_{\vec{v}}(A_i) \neq \pi_{\vec{v}}(A_j)$ for any $i \neq j$. Thus, we have
\[ |A+B| = \sum_{i=1}^{r}|A_i + B| \geq \sum_{i=1}^{r} (|A_i| + |B| -1 ) = |A| + r |B| - r. \qedhere \]
\end{proof}

With this lemma in hand, we now turn to our proof of Theorem $\ref{mn2}$ in the case when $r_2 \geq d$.

\begin{theorem} \label{rszgn}
Let $r_1, r_2, d \in \mathbb{N}$ satisfy $r_1,r_2 \geq d\geq 2$ and let $A,B$ be finite, non-empty subsets of $\mathbb{R}^d$ such that $|A| \geq |B|$ and $\dim(A+B) = d$. Moreover, let $\vec{v} \in \mathbb{R}^d \setminus \{0\}$ satisfy $\pi_{\vec{v}}(A) = r_1$ and $\pi_{\vec{v}}(B) = r_2$. Then 
\[ |A+B| \geq |A| + \bigg( d+1 - \frac{1}{r_1 - d + 2} - \frac{1}{r_2 - d + 2} \bigg)|B| - (d-1)(r_1 + r_2).  \]
\end{theorem}

\begin{proof}
Our base case is when $d=2$, which can be resolved by Lemma $\ref{gsne2}$. In order to see this, we apply Lemma $\ref{gsne2}$ in our setting to deduce that
\begin{align*}
 |A+B| & \geq  |A| (1 + (r_2-1)/r_1) + |B|(1 + (r_1 - 1)/r_2) - (r_1 + r_2) +1 \\
 & \geq |A| + |B|( 1 + r_2/r_1 + r_1/r_2 - 1/r_1 - 1/r_2 ) - (r_1 + r_2) \\
 & \geq |A| + |B|(3 - 1/r_1 - 1/r_2) - (r_1 + r_2),
 \end{align*}
where the last inequality follows from noting the fact that $r_1/r_2 + r_2/r_1 \geq 2$ for every $r_1, r_2 >0$. This is precisely the desired estimate, and so, from here on, we may assume that $d \geq 3$. 
\par

Applying Lemmas $\ref{realin}$ and $\ref{cmp1}$, we see that it suffices to reduce to the case when $A$ and $B$ are subsets of $\mathbb{Z}^d$ satisfying the properties that $A'$ and $B'$ satisfy in the conclusion of Lemma $\ref{cmp1}$. Writing $U = \pi_d(A)$ and $V = \pi_d(B)$, we first consider the case when $U \subseteq \mathbb{Z}^{d-2} \times \{0\}$ and $V \setminus (\mathbb{Z}^{d-2} \times \{0\}) = \{ e_{d-1}\}$. In this case, we set $V_1 = V \setminus \{e_{d-1}\}$. Note that $V_1^{B}$ and $U^{A}$ are finite subsets of $\mathbb{R}^{d-2} \times \{0\}\times \mathbb{R}$ such that $\dim(U^{A} + V_1^{B}) = d-1$ and $\pi_{d}(V_1^B) = r_2 - 1$. Since $|U^A| = |A| \geq |B| \geq |V_1^B|$, the inductive hypothesis implies that
\[ |U^A + V_1^B| \geq |U^A| + \bigg( d- \frac{1}{r_1 - d + 3} - \frac{1}{r_2 - d + 2} \bigg)|V_1^B| - (d-2)(r_1 + r_2 -1). \]
Moreover, by Lemma $\ref{d12}$, we have
\[ |U^A + e_{d-1}^B| \geq |U^A| + r_1 |e_{d-1}^B| - r_1 ,\]
and so, upon combining the above two bounds with the fact that
\[ |A+B| = |U^A + V_1^B| + |U^A + e_{d-1}^B|, \]
we get that
\[ |A+B| \geq 2|A| + \bigg( d - \frac{1}{r_1 - d + 3} - \frac{1}{r_2 - d + 2} \bigg)|V_1^B| + r_1 |e_{d-1}^B| - (d-1)(r_1 + r_2). \]
We may now use the facts that $|A| \geq |B| = |V_1^B| + |e_{d-1}^B|$ and that $r_1 \geq d$ to close the inductive loop. 
\par

The next case we consider is when $U \setminus (\mathbb{Z}^{d-2} \times \{0\}) = \{e_{d-1}\}$ while $V \subseteq \mathbb{Z}^{d-2} \times \{0\}$. Here, we write $U_1 = U \setminus \{ e_{d-1}\}$. We have to further divide into two cases depending on how large $|U_1^A|$ is compared to $|B|$. In particular, we first suppose that $|U_1^A| \geq |B|$, in which case, we may use the inductive hypothesis to deduce that
\[ |U_1^A + V^B| \geq |U_1^A| +  \bigg( d - \frac{1}{r_1 - d + 2} - \frac{1}{r_2 - d + 3} \bigg)|V^B| - (d-2)(r_1 + r_2-1).  \]
Combining this with the fact that 
\[ |e_{d-1}^A + V^B| \geq |e_{d-1}^A| + |V^B| - 1, \]
and that
\begin{equation} \label{rvn}
|A+ B| \geq |U_1^A + V^B| + |e_{d-1}^A + V^B| ,
\end{equation}
we are done.
\par

The more problematic subcase is when $|U_1^A| < |B|$, in which case, the inductive hypothesis implies that
\[ |U_1^A + V^B| \geq |V^B| +  \bigg( d - \frac{1}{r_1 - d + 2} - \frac{1}{r_2 - d + 3} \bigg)|U_1^A| - (d-2)(r_1 + r_2-1).  \]
But now, we may employ Lemma $\ref{d12}$ to obtain the bound
\[ |e_{d-1}^A + V^B| \geq r_2 |e_{d-1}^A| + |V^B| - r_2 \geq d |e_{d-1}^A| + |V^B| - r_2, \]
which, when amalgamated with the preceding inequality and $\eqref{rvn}$, gives us
\[ |A+B| \geq \bigg( d - \frac{1}{r_1 - d + 2} - \frac{1}{r_2 - d + 3} \bigg)|A|+ 2|B| - (d-1)(r_1 + r_2). \]
Note that since $r_1, r_2 \geq d \geq 2$ and $|A| \geq |B|$, we get
\[ \bigg( d - \frac{1}{r_1 - d + 2} - \frac{1}{r_2 - d + 3} \bigg)|A|+ 2|B| \geq |A| + \bigg( d+1 - \frac{1}{r_1 - d + 2} - \frac{1}{r_2 - d + 3} \bigg)|B| , \]
and so, we are done.
\par

Finally, the last case is when $U \setminus (\mathbb{Z}^{d-2} \times \{0\}) =  \{e_{d-1}\}$ and $U = V$. In this case, we write $U_1 = U \setminus \{e_{d-1}\}$ and $V_1 = V \setminus \{e_{d-1}\}$. Note that $U=V$ implies that $r_1 = |U| = |V| = r_2$. Moreover, the case when $r_1 = r_2 = d$ can be derived from Ruzsa's result, whence, we may assume that $r_1 = r_2 \geq d+1$. We first suppose that $|U_1^A| \geq |V_1^B|$, in which case, the inductive hypothesis suggests that
\[ |U_1^A + V_1^B| \geq |U_1^A| + \bigg( d - \frac{1}{r_1 - d + 2} - \frac{1}{r_2 - d + 2} \bigg)|V_1^B| - (d-2)(r_1 + r_2).\]
Moreover, by Lemma $\ref{d12}$, we have
\[ |U_1^A + e_{d-1}^B| \geq |U_1^A| + r_1 |e_{d-1}^B| - r_1 \geq |V_1^B| + (d+1)|e_{d-1}^B| - r_1, \]
and
\[ |e_{d-1}^A + e_{d-1}^B| \geq |e_{d-1}^A| + |e_{d-1}^B| - 1. \]
As before, we may then deduce that
\begin{align*}
 |A+B| 
 &  \geq |U_1^A + V_1^B| + |U_1^A + e_{d-1}^B| + |e_{d-1}^A + e_{d-1}^B| \\
 & \geq |A| + \bigg( d +1- \frac{1}{r_1 - d + 2} - \frac{1}{r_2 - d + 2} \bigg)|V_1^B|+ (d+1) |e_{d-1}^B| - (d-1)(r_1 + r_2), 
 \end{align*}
 in which case, we are done. 
 \par
 
We now consider the second subcase, that is, when $|U_1^A| < |V_1^B|$. In this case, the inductive hypothesis implies that
\[ |U_1^A + V_1^B| \geq  \bigg( d - \frac{1}{r_1 - d + 2} - \frac{1}{r_2 - d + 2} \bigg)|U_1^A| +|V_1^B| - (d-2)(r_1 + r_2).\]
Furthermore, by Lemma $\ref{d12}$, we have
\[ | e_{d-1}^A+ V_1^B|  \geq |V_1^B| + r_2 |e_{d-1}^A| - r_2 \geq |U_1^A| + (d+1) |e_{d-1}^A| - r_2. \]
Thus, we get
\begin{align*}
|A+B|
\geq & |U_1^A + V_1^B| + | e_{d-1}^A+ V_1^B| + |e_{d-1}^A + e_{d-1}^B| \\
\geq &  \bigg( d +1 - \frac{1}{r_1 - d + 2} - \frac{1}{r_2 - d + 2} \bigg)|A| + |B| - (d-1)(r_1 + r_2) \\
\geq &  |A| + \bigg( d +1 - \frac{1}{r_1 - d + 2} - \frac{1}{r_2 - d + 2} \bigg) |B| - (d-1)(r_1 + r_2),
\end{align*}
which finishes the proof of our theorem. 
\end{proof}

\section{Lower bounds for sumsets II}

We dedicate this section to proving Theorem $\ref{mn2}$ in the case when $r_2 < d$. Moreover, the subcase when $|B| = 1$ follows trivially, and so, we may assume that $|B| \geq 2$.

\begin{theorem} \label{rszgn2}
Let $r_1, r_2, d, k\in \mathbb{N}$ satisfy $r_1 \geq d> r_2, k \geq 1$. Let $A,B$ be finite, non-empty subsets of $\mathbb{R}^d$ such that $|A| \geq |B| \geq 2$ and $\dim(A+B) = d$ and $\dim(B) =k$, and let $\vec{v} \in \mathbb{R}^d \setminus \{0\}$ satisfy $|\pi_{\vec{v}}(A)| = r_1$ and $|\pi_{\vec{v}}(B)| = r_2$. Then, we have
\[ |A+B| \geq |A| + \bigg( d+1 - \frac{1}{r_1 - d + 2} - \frac{1}{r_2 - k + 2} \bigg)|B| - (d-1)(r_1 + r_2) . \]
\end{theorem}

\begin{proof}
We may assume that $r_2 \leq |B| -1$ since otherwise, we can employ $\eqref{runit}$ to deduce our result. The fact that $r_2 \leq |B|-1$ would further imply that there exist $\vec{b}_1, \vec{b}_2 \in B$ such that $\vec{b}_1 - \vec{b}_2  = \lambda \cdot \vec{v}$ for some $\lambda \neq 0$. We now translate $B$ appropriately to ensure that $\vec{b}_2 = 0$, and moreover, since our hypothesis remains invariant under dilations of $\vec{v}$, we may choose $\vec{v}$ to lie in the set $B$.  Next, we note that when $d = 2$, we must have $r_2 = 1$, in which case, we can apply Lemma $\ref{d12}$ to procure the desired bound, and so, we may suppose that $d \geq 3$. 
 \par
 
Note that when $k=1$, the fact that $0, \vec{v} \in B$ would imply that $r_2 = 1$, in which case, Lemma $\ref{d12}$ yields the bound
\[ |A+B|  \geq |A| + r_1 |B| - r_1. \]
This delivers the required result, and so, it suffices to assume that $k \geq 2$. We also have that $k \leq r_2$, since whenever $r_2 \geq 2$, then $r_2$ parallel lines are contained in an affine subspace of dimension at most $r_2$. Moreover, we may apply Lemmas $\ref{ur3}$ and $\ref{tmb}$ to ensure that the sets $A$ and $B$ satisfy the properties that $A'$ and $B'$ satisfy in the conclusion of Lemma $\ref{tmb}$. 
\par

Writing $X = A \cap (\mathbb{Z}^k \times \{0\}^{d-k})$, we see that
\[ |A+B| = |X+B| + \sum_{i=k+1}^{d} |e_{i} + B| \geq |X+B| + (d-k)|B|. \]
Moreover, we see that $|\pi_{1}(X)| = r_1 - d +k \geq k$, and so, we are now ready to apply Theorem $\ref{rszgn}$. In particular, we claim that the latter gives us
\[ |X+B| \geq  |X| + \bigg( k+1 - \frac{1}{r_1 - d + 2} - \frac{1}{r_2 - k + 2} \bigg)|B| - k(r_1 +r_2 - d +k ) -k(d-k).  \]
In order to see this, note that if $|X| \geq |B|$, then the conclusion of Theorem $\ref{rszgn}$ yields an even stronger bound than the above estimate. On the other hand, if $|B| \geq |X|$, then upon observing that $|B| - |X| = |B| - |A| + (d-k) \leq d-k$, we see that we are able to interchange the coefficients of $|B|$ and $|X|$ in  the conclusion of Theorem $\ref{rszgn}$ by incurring a further loss of the amount $k(d-k)$, which, in turn, delivers the desired claim. Finally, amalgamating the claimed estimate with the preceding discussion, we get that
\[ |A+B| \geq |A| + \bigg( d+1 - \frac{1}{r_1 - d + 2} - \frac{1}{r_2 - k + 2} \bigg)|B| - (d-1)(r_1 + r_2) , \]
and so, we are done.
\end{proof}


%


\section*{Acknowledgments} 
The author is grateful to Misha Rudnev for many helpful discussions, and to the anonymous referee for many helpful comments. Finally, as the author was finishing the paper, it came to his attention that, using related methods, Conlon and Lim \cite{CL2021} have independently obtained an estimate akin to Theorem $\ref{diff}$. They are able to obtain a sharp $O_{d}(1)$ error term, instead of our $O_{d}(|A|^{1- \delta})$ term, for all sufficiently large subsets $A$ of $\mathbb{R}^d$.

\bibliographystyle{amsplain}


\begin{dajauthors}
\begin{authorinfo}[am]
  Akshat Mudgal\\

Mathematical Institute, \\
University of Oxford, \\
Oxford OX2 6GG, UK\\
Akshat.Mudgal\imageat{}maths\imagedot{}ox\imagedot{}ac\imagedot{}uk \\
  \url{https://sites.google.com/view/akshatmudgal}
\end{authorinfo}
\end{dajauthors}

\end{document}